\documentclass[leqno,12pt]{article} 
\usepackage{amsmath, amssymb}
\usepackage{amsthm} 
\theoremstyle{plain} 
\newtheorem{theorem}{\indent\sc Theorem}[section]
\newtheorem{lemma}[theorem]{\indent\sc Lemma}
\newtheorem{corollary}[theorem]{\indent\sc Corollary}
\newtheorem{proposition}[theorem]{\indent\sc Proposition}

\theoremstyle{definition} 
\newtheorem{definition}[theorem]{\indent\sc Definition}
\newtheorem{remark}[theorem]{\indent\sc Remark}

\newcommand\on{\operatorname}

\newcommand\Ric{\on{Ric}}
\newcommand\scal{\on{scal}}

\title{$\alpha$-connections in generalized geometry}
\author{Adara M. Blaga and Antonella Nannicini}
\date{}
\begin{document}

\maketitle

\markboth{{\small\it {\hspace{4cm} $\alpha$-connections in generalized geometry}}}{\small\it{$\alpha$-connections in generalized geometry
\hspace{4cm}}}


\begin{abstract}
We consider a family of $\alpha$-connections defined by a pair of generalized dual quasi-statistical connections $(\hat{\nabla},\hat{\nabla}^*)$ on the generalized tangent bundle $(TM\oplus T^*M, \check{h})$ and determine their curvature, Ricci curvature and scalar curvature.
Moreover, we provide the necessary and sufficient condition for $\hat \nabla^*$ to be an equiaffine connection and we prove that if $h$ is symmetric and $\nabla h=0$, then $(TM\oplus T^*M, \check{h}, \hat{\nabla}^{(\alpha)}, \hat{\nabla}^{(-\alpha)})$ is a conjugate Ricci-symmetric manifold. Also, we characterize the integrability of a generalized almost product, of a generalized almost complex and of a generalized metallic structure w.r.t. the bracket defined by the $\alpha$-connection. Finally we study $\alpha$-connections defined by the twin metric of a pseudo-Riemannian manifold, $(M,g)$, with a non-degenerate $g$-symmetric $(1,1)$-tensor field $J$ such that $d^\nabla J=0$, where $\nabla$ is the Levi-Civita connection of $g$.
\end{abstract}

\bigskip

\section{Introduction}

A bridge between Differential Geometry, Information Geometry and Theoretical Physics, statistical manifolds were firstly considered by Amari \cite{am}.
Geometrical aspects of statistical structures such as invariance, properties of submanifolds in statistical manifolds etc. have been lately studied. Statistical manifolds are also connected to Hessian manifolds \cite{sh}. Basically, a statistical structure on a smooth manifold $M$ consists of a pseudo-Riemannian metric $g$ and a torsion-free affine connection $\nabla$ such that $\nabla g$ is a Codazzi tensor field. To every statistical structure $(g,\nabla)$, one can naturally associate a dual statistical structure $(g,\nabla^*)$. Such a dualistic pair $(g;\nabla,\nabla^*)$ defines a family of connections, called $\alpha$-connections \cite{am}, which plays a significant role in Information Geometry.\\

In the larger framework of Generalized Geometry, we consider a family of $\alpha$-connections defined by a pair of generalized dual quasi-statistical connections $(\hat{\nabla},\hat{\nabla}^*)$ on the generalized tangent bundle $(TM\oplus T^*M, \check{h})$ and determine their curvature, Ricci curvature and scalar curvature.
Moreover, we provide the necessary and sufficient condition for $\hat \nabla^*$ to be an equiaffine connection and we prove that if $h$ is symmetric and $\nabla h=0$, then $(TM\oplus T^*M, \check{h}, \hat{\nabla}^{(\alpha)}, \hat{\nabla}^{(-\alpha)})$ is a conjugate Ricci-symmetric manifold. Also, we characterize the integrability of a generalized almost product, of a generalized almost complex and of a generalized metallic structure with respect to the bracket defined by the $\alpha$-connection, finding conditions under which the concept of integrability is $\alpha$-invariant. In the last section, we focus on $\alpha$-connections defined by the twin metric of a pseudo-Riemannian manifold, $(M,g)$, with a non-degenerate $g$-symmetric $(1,1)$-tensor field $J$ such that $d^\nabla J=0$, where $\nabla$ is the Levi-Civita connection of $g$.

\section{Preliminaries}

\subsection{Quasi-statistical structures}

Let $M$ be a smooth manifold and let $TM$ be its tangent bundle. Let $g$ be a pseudo-Riemannian metric and let $\nabla$ be a torsion-free affine connection on $M$. $(g,\nabla)$ is called a \textit{statistical structure} on $M$ if $$({\nabla}_Xg)(Y,Z)=({\nabla}_Yg)(X,Z),$$ for any $X,Y,Z\in C^{\infty}(TM)$.

We can extend this definition and say that $(h,\nabla)$ is a \textit{quasi-statistical structure} on $M$ if $h$ is a non-degenerate $(0,2)$-tensor field, $\nabla$ is an affine connection with torsion tensor $T^\nabla$ and $d^{\nabla}h=0$, where $$(d^{\nabla}h)(X,Y,Z):=({\nabla}_Xh)(Y,Z)-({\nabla}_Yh)(X,Z)+h(T^{\nabla}(X,Y),Z),$$ for any $X,Y,Z\in C^{\infty}(TM)$.

\subsection{Dualistic structures}

Let $M$ be a smooth manifold and let $h$ be a non-degenerate $(0,2)$-tensor field on $M$.
\begin{definition} Two affine connections $\nabla$ and $\nabla^*$ on $M$ are said to be \textit{dual connections} with respect to  $h$ if
$$X(h(Y,Z))=h(\nabla_XY,Z)+h(Y,\nabla^*_XZ),$$
for any $X,Y,Z\in C^{\infty}(TM)$ and we call $(h;\nabla,\nabla^*)$ a \textit{dualistic structure}.
\end{definition}

\begin{lemma} If $h$ is a non-degenerate symmetric or skew-symmetric $(0,2)$-tensor field on $M$ and $\nabla$ is an affine connection on $M$, then the dual connection $\nabla^*$ satisfies:
$$\nabla^*_XY=\nabla_XY+h^{-1}((\nabla_Xh)(Y)), \ \ \nabla^*_X\beta=\nabla_X\beta-(\nabla_Xh)(h^{-1}(\beta)),$$
for any $X,Y\in C^{\infty}(TM)$ and $\beta \in C^{\infty}(T^*M)$.

Moreover, if $(h,\nabla)$ is a quasi-statistical structure, then the dual connection $\nabla^*$ satisfies:
$$T^{\nabla^*}=0, \ \ \nabla^* h=-\nabla h,$$
therefore, $(h,\nabla^*)$ is always a statistical structure.
\end{lemma}
\begin{proof}  From $X(h(Y,Z))=h(\nabla_XY,Z)+h(Y,\nabla^*_XZ),$
we get:
$$({\nabla_X}h)(Y,Z)=X(h(Y,Z))-h(\nabla_XY,Z)-h(Y,\nabla_XZ)=$$
$$=h(\nabla_XY,Z)+h(Y,\nabla^*_XZ)-h(\nabla_XY,Z)-h(Y,\nabla_XZ)=$$
$$=h(Y,\nabla^*_XZ-\nabla_XZ),$$
for any $X,Y,Z\in C^{\infty}(TM).$
Therefore:
$$h(\nabla^*_XZ-\nabla_XZ)=({\nabla_X}h)(Z),$$
for any $X,Z\in C^{\infty}(TM)$. Then the first statement.

Let us compute $ \nabla^*_X\beta$. We get:
$$ (\nabla^*_X\beta)(Y)=X(\beta (Y))-\beta(\nabla^*_XY)=X(\beta (Y))-\beta (\nabla_X Y)-\beta ( h^{-1}((\nabla_Xh)(Y)))=$$
$$= (\nabla_X\beta)(Y)-\beta ( h^{-1}((\nabla_Xh)(Y))).$$

Let us suppose $\beta=h(Z)$ for some $Z\in C^{\infty}(TM)$. Then we obtain:
$$\beta ( h^{-1}((\nabla_Xh)(Y)))=h(Z)( ( h^{-1}((\nabla_Xh)(Y))))=h(Z, ( h^{-1}((\nabla_Xh)(Y))))=$$
$$=\pm h(( h^{-1}((\nabla_Xh)(Y)),Z)=\pm ((\nabla_Xh)(Y))(Z)=\pm ((\nabla_Xh)(Y))(h^{-1}(\beta))=$$
$$=\pm((\nabla_Xh)(Y,h^{-1}(\beta))=(\nabla_Xh)(h^{-1}(\beta),Y)=((\nabla_Xh)(h^{-1}(\beta)))(Y),$$
where the sign $+$ is for $h$ symmetric and $ - $ for $h$ skew-symmetric. Then the second statement is proved.

Now let us suppose that $(h,\nabla)$ is a quasi-statistical structure and let $T^*$ be the torsion of the dual connection $\nabla^*$,
$$T^*(X,Y):=\nabla^*_X Y-\nabla^*_Y X-[X,Y].$$
We have:
$$T^*(X,Y)={\nabla}_X Y-{\nabla}_Y X-[X,Y]+h^{-1}((\nabla_X h)(Y)-(\nabla_Y h)(X))=$$
$$=T^{\nabla}(X,Y)-h^{-1}(h(T^{\nabla}(X,Y)))=0.$$
Finally:
$$({{\nabla}^*_X}h)(Y,Z)=X(h(Y,Z))-h({\nabla}^*_XY,Z)-h(Y,{\nabla}^*_XZ)=$$
$$=({\nabla_X}h)(Y,Z)-({\nabla_X}h)(Y,Z)-h(Y,h^{-1}((\nabla_X h)(Z)))=\pm(\nabla_X h)(Z,Y)=-(\nabla_X h)(Y,Z).$$
Therefore $\nabla^* h=-\nabla h$ and the proof is complete.
\end{proof}
\subsection{Some geometrical structures on $TM\oplus T^*M$}

Let $TM\oplus T^*M$ be the generalized tangent bundle of $M$. On $TM\oplus T^*M$, we consider the natural indefinite metric
\begin{equation}\label{m1}
<X+\eta,Y+\beta>:=-\frac{1}{2}(\eta(Y)+\beta(X))
\end{equation}
and the natural symplectic structure
\begin{equation}\label{m2}
(X+\eta,Y+\beta):=-\frac{1}{2}(\eta(Y)-\beta(X)),
\end{equation}
for all $X,Y\in C^{\infty}(TM)$ and $\eta,\beta\in C^{\infty}(T^*M).$

If $h$ is a non-degenerate $(0,2)$-tensor field on $M$, we define the bilinear form, $\check h$, on $TM\oplus T^*M$ by:
\begin{equation}\label{m3}
\check h(X+\eta, Y+\beta):=h(X,Y)+h(h^{-1}(\eta),h^{-1}(\beta)),
\end{equation}
for all $X,Y\in C^{\infty}(TM)$ and $\eta,\beta\in C^{\infty}(T^*M)$.

Furthermore, given an affine connection $\nabla$ on $M$, we define the affine connections $\hat{\nabla}$ and $\check{\nabla}$ on $TM \oplus T^*M$ by:
\begin{equation}\label{m5}
\hat{\nabla}_{X+\eta}(Y+\beta):=\nabla_XY+h(\nabla_X(h^{-1}(\beta)))
\end{equation}
and
\begin{equation}\label{m4}
\check{\nabla}_{X+\eta}(Y+\beta):=\nabla_XY+\nabla_X\beta.
\end{equation}

Remark that $\hat \nabla =\check \nabla$ if and only if $\nabla h=0$.

Finally we define the bracket $[\cdot,\cdot]_{\nabla}$:
\begin{equation}
[X+\eta,Y+\beta]_{\nabla}:=[X,Y]+\nabla_X\beta-\nabla_Y\eta,
\end{equation}
for all $X,Y\in C^{\infty}(TM)$ and $\eta,\beta\in C^{\infty}(T^*M)$.

\section{$\alpha$-connections in generalized geometry}

\subsection{Generalized quasi-statistical structures}

In our previous paper \cite{blanan} we introduced the concept of generalized quasi-statistical structure.
\begin{definition} Let $h$ be a non-degenerate $(0,2)$-tensor field and let $\nabla$ be an affine connection on $M$. Let $\hat{\nabla}$ be the induced connection on $TM \oplus T^*M$ and let $\hat{h}$ be a non degenerate bilinear form on $TM\oplus T^*M$. Then $(\hat{h},\hat{\nabla})$ is called a \textit{generalized quasi-statistical structure on $TM\oplus T^*M$} if:
$$(d^{\hat{\nabla}}\hat{h})(\sigma,\tau,\nu):=(\hat{\nabla}_{\sigma}\hat{h})(\tau,\nu)-(\hat{\nabla}_{\tau}\hat{h})(\sigma,\nu)+
\hat{h}(T^{\hat{\nabla}}(\sigma,\tau),\nu),$$
for any $\sigma,\tau,\nu\in C^{\infty}(TM\oplus T^*M)$, where $T^{\hat{\nabla}}(\sigma,\tau):=\hat{\nabla}_{\sigma}\tau-\hat{\nabla}_{\tau}\sigma-[\sigma,\tau]_{\nabla}$.
\end{definition}

We proved the followings:
\begin{theorem}\cite{blanan}
Let $\hat{h}$ given by (\ref{m1}) or (\ref{m2}). Then $(\hat{h},\hat{\nabla})$ is a generalized quasi-statistical structure on $TM\oplus T^*M$ if and only if $(M,{h},\nabla)$ is a quasi-statistical manifold.
\end{theorem}
Moreover:
\begin{proposition}\cite{blanan}
Let $(M,h,\nabla)$ be a quasi-statistical manifold and let $(\hat{h},\hat{\nabla})$ be the generalized quasi-statistical structure on $TM\oplus T^*M$,
with $\hat{h}$ given by (\ref{m1}) or (\ref{m2}). Then the generalized dual quasi-statistical connection, ${\hat {\nabla}}^*$, defined by:
$$\hat h(Y+\beta, {\hat {\nabla}}^*_{X+\eta}(Z+\gamma))=X(\hat h(Y+\beta,Z+\gamma))-\hat h({\hat \nabla}_{X+\eta}(Y+\beta),Z+\gamma),$$
for all $X,Y,Z \in C^{\infty}(TM)$ and $\eta, \beta, \gamma \in C^{\infty}(T^*M)$, is given by:
$${\hat {\nabla}}^*_{X+\eta}(Z+\gamma)=h^{-1}({\nabla}_X(h(Z)))+{\nabla}_X \gamma.$$
\end{proposition}

Also, if $h$ is a non-degenerate, symmetric or skew-symmetric $(0,2)$-tensor field on $M$, then:
\begin{proposition}\cite{blanan}
$(\check{h}, \hat{\nabla})$ is a generalized quasi-statistical structure on $TM\oplus T^*M$ if and only if $(M,h,\nabla)$ is a quasi-statistical manifold.
\end{proposition}

\begin{proposition}\cite{blanan}
Let $(M,h,\nabla)$ be a quasi-statistical manifold and let $(\check{h},\hat{\nabla})$ be the generalized quasi-statistical structure induced on $TM\oplus T^*M$. Then the generalized dual quasi-statistical connection, ${{\hat {\nabla}}^*}_{\check h}$, defined by:
$$ \check h (Y+\beta, ({{{\hat{\nabla}}^*}}_{\check h})_{X+\eta} (Z+\gamma))=X(\check h(Y+\beta,Z+\gamma))-\check h({\hat \nabla}_{X+\eta}(Y+\beta),Z+\gamma),$$
for all $X,Y,Z \in C^{\infty}(TM)$ and $\eta, \beta, \gamma \in C^{\infty}(T^*M)$, coincides with ${\hat{\nabla}}^*$.
\end{proposition}

\begin{proposition}\cite{blanan}
Let $\nabla$ be a torsion-free affine connection on $M$ and let $(\hat{h}, \check{\nabla})$ be the generalized quasi-statistical structure on $TM\oplus T^*M$,
with $\hat{h}$ given by (\ref{m1}) or (\ref{m2}). Then $\check{\nabla}$ and its generalized dual quasi-statistical connection, ${\check{\nabla}}^*$, coincide.
\end{proposition}

\begin{proposition}\cite{blanan}
Let $(M,h,\nabla)$ be a quasi-statistical manifold with $\nabla$ a torsion-free affine connection, $h$ a $\nabla$-parallel $(0,2)$-tensor field on $M$ and let $(\check{h},\check{\nabla})$ be the generalized quasi-statistical structure on $TM\oplus T^*M$, with $\check{h}$ given by (\ref{m3}). Then $\check{\nabla}$ and its generalized dual quasi-statistical connection, ${{\check {\nabla}}^*}_{\check h}$, coincide.
\end{proposition}

\subsection{$\alpha$-connections}

In \cite{blanan} we obtained the following pairs of dual connections with respect to $\hat{h}$
given by (\ref{m1}) or (\ref{m2}) and to $\check{h}$ given by (\ref{m3}):
$$(\hat{h}; \hat{\nabla},\hat{\nabla}^*), \ \ (\check{h}; \hat{\nabla},\hat{\nabla}^*),$$
where $\hat{\nabla}$ is given by (\ref{m5}) and $\check{\nabla}$ is given by (\ref{m4}). Remark that the dual of the generalized dual quasi-statistical connection with respect to $\hat{h}$ coincides with the initial connection, and with respect to $\check{h}$ coincides with the initial connection if $\check{h}$ is symmetric or skew-symmetric.

For the dualistic structure $(\check{h}; \hat{\nabla},\hat{\nabla}^*)$, we consider a family of connections, $\{ \hat{\nabla}^{(\alpha)}\}$, on $TM\oplus T^*M$, for any $\alpha \in \mathbb{R}$, called \textit{$\alpha$-connections}:
$$\hat{\nabla}^{(\alpha)}:=\frac{1+\alpha}{2}\hat{\nabla}+\frac{1-\alpha}{2}\hat{\nabla}^*.$$

We immediately have that $\hat{\nabla}^{(1)}=\hat \nabla$, $\hat{\nabla}^{(-1)}={\hat \nabla}^*$ and for $\alpha=0$, \textit{the average connection} of $\hat{\nabla}$ and $\hat{\nabla}^*$ is:
$$\hat{\nabla}^{(0)}_{X+\eta}(Y+\beta)=\frac{1}{2}\check{\nabla}_{X+\eta}(Y+\beta)+\frac{1}{2}\{h^{-1}((\nabla_Xh)(Y))-(\nabla_Xh)(h^{-1}(\beta))\},$$
for all $X,Y\in C^{\infty}(TM)$ and $\eta,\beta\in C^{\infty}(T^*M)$.

\bigskip
\begin{proposition}
The torsion of the $\alpha$-connection is given by:
$$T^{\hat{\nabla}^{(\alpha)}}(X+\eta,Y+\beta)=T^{\nabla}(X,Y)+$$
$$-\frac{1+\alpha}{2}\{(\nabla_Xh)(h^{-1}(\beta))-(\nabla_Yh)(h^{-1}(\eta))\}+$$
$$+\frac{1-\alpha}{2}\{h^{-1}((\nabla_Xh)(Y))-h^{-1}((\nabla_Yh)(X))\},$$
for any $X,Y\in C^{\infty}(TM)$ and $\eta,\beta\in C^{\infty}(T^*M)$.
\end{proposition}
\begin{proof}
$$T^{\hat{\nabla}^{(\alpha)}}(X+\eta,Y+\beta):=
\hat{\nabla}^{(\alpha)}_{X+\eta}(Y+\beta)-\hat{\nabla}^{(\alpha)}_{Y+\beta}(X+\eta)-[X+\eta,Y+\beta]_{\nabla}=$$
$$=T^{\nabla}(X,Y)+\frac{1+\alpha}{2}\{h(\nabla_X(h^{-1}(\beta)))-\nabla_X\beta
-h(\nabla_Y(h^{-1}(\eta)))+\nabla_Y\eta\}+$$
$$+\frac{1-\alpha}{2}\{h^{-1}(\nabla_X(h(Y)))-\nabla_XY-h^{-1}(\nabla_Y(h(X)))+\nabla_YX\}=$$
$$=T^{\nabla}(X,Y)-\frac{1+\alpha}{2}\{(\nabla_Xh)(h^{-1}(\beta))-(\nabla_Yh)(h^{-1}(\eta))\}+$$
$$+\frac{1-\alpha}{2}\{h^{-1}((\nabla_Xh)(Y))-h^{-1}((\nabla_Yh)(X))\},$$
for any $X,Y\in C^{\infty}(TM)$ and $\eta,\beta\in C^{\infty}(T^*M)$.
\end{proof}

\begin{remark}
On the generalized tangent bundle, the family of $\alpha$-connections can be constructed in two ways, which coincide if $h$ is symmetric or skew-symmetric.
Precisely, let $(h,\nabla)$ be a quasi-statistical structure on $M$ and let $(\check{h},\hat{\nabla})$ be the generalized quasi-statistical structure on $TM\oplus T^*M$ induced by $(h,\nabla)$.

For the quasi-statistical structure $(h,\nabla)$ and $\nabla^*$ the dual connection of $\nabla$, we consider the family of $\alpha$-connections on $M$:
$$
\nabla^{(\alpha)}:=\frac{1+\alpha}{2}\nabla+\frac{1-\alpha}{2}\nabla^*.
$$
Then the dual connection of $\nabla^{(\alpha)}$ is $\nabla^{(-\alpha)}$.

For $(h,\nabla^{(\alpha)})$ and $\nabla^{(-\alpha)}$ the dual connection of $\nabla^{(\alpha)}$, let $(\check{h},\widehat{\nabla^{(\alpha)}})$ be the generalized structure on $TM\oplus T^*M$ induced by $(h,\nabla^{(\alpha)})$.

For the generalized quasi-statistical structure $(\check{h},\hat{\nabla})$ and $\hat{\nabla}^*$ the dual connection of $\hat{\nabla}$, we consider the family of $\alpha$-connections on $TM\oplus T^*M$:
$$\hat{\nabla}^{(\alpha)}:=\frac{1+\alpha}{2}\hat{\nabla}+\frac{1-\alpha}{2}\hat{\nabla}^*.$$
Then the dual connection of $\hat{\nabla}^{(\alpha)}$ is $\hat{\nabla}^{(-\alpha)}$.

Then:
$$(\widehat{\nabla^{(\alpha)}}-\hat{\nabla}^{(\alpha)})(X+\eta,Y+\beta)=
\frac{1-\alpha}{2}\{h^{-1}((\nabla_Xh)(Y))-h^{-1}(\nabla_X(h(Y)))+\nabla_XY+$$$$+
(\nabla_Xh)(h^{-1}(\beta))+h(\nabla_X(h^{-1}(\beta)))-\nabla_X\beta\}=0.$$
\end{remark}

\begin{remark}
If $(h,\nabla)$ is a quasi-statistical structure on $M$ and $h$ is symmetric or skew-symmetric, then:

i) $$T^{\nabla^{(\alpha)}}=\frac{1+\alpha}{2}T^{\nabla}, \ \ \nabla^{(\alpha)} h=\alpha\nabla h, \ \ (d^{\nabla^{(\alpha)}}h)(X,Y,Z)=\frac{1-\alpha}{2}h(T^{\nabla}(X,Y),Z),$$ for any $X,Y,Z\in C^{\infty}(TM)$ and any $\alpha\in \mathbb{R}$, therefore $(h,\nabla^{(\alpha)})$ (with $\alpha\neq 1$) is a statistical structure if and only if $T^{\nabla}=0$, i.e. if and only if $(h,\nabla)$ is a statistical structure;

ii) $$(d^{\hat{\nabla}}\check h)(X+\eta,Y+\beta,Z+\gamma)=(d^{\nabla} h)(X,Y,Z),$$ for any $X,Y,Z\in C^{\infty}(TM)$ and $\eta,\beta,\gamma\in C^{\infty}(T^*M)$, therefore $(\check h, \hat \nabla)$ is a generalized quasi-statistical structure on $TM\oplus T^*M$;

iii) $$(d^{\widehat{\nabla^{(\alpha)}}}\check h)(X+\eta,Y+\beta,Z+\gamma)=(d^{\nabla^{(\alpha)}} h)(X,Y,Z)=\frac{1-\alpha}{2}h(T^{\nabla}(X,Y),Z),$$ for any $X,Y,Z\in C^{\infty}(TM)$, $\eta,\beta,\gamma\in C^{\infty}(T^*M)$ and any $\alpha\in \mathbb{R}$, therefore $(\check h, \widehat{\nabla^{(\alpha)}}=\hat{\nabla}^{(\alpha)})$ (with $\alpha\neq 1$) is a generalized quasi-statistical structure on $TM\oplus T^*M$ if and only if $T^{\nabla}=0$, i.e. if and only if $(h,\nabla)$ is a statistical structure on $M$. In this case: $$T^{\hat{\nabla}^{(\alpha)}}(X+\eta,Y+\beta)=-\frac{1+\alpha}{2}\{(\nabla_Xh)(h^{-1}(\beta))-(\nabla_Yh)(h^{-1}(\eta))\},$$ for any $X,Y\in C^{\infty}(TM)$ and $\eta,\beta\in C^{\infty}(T^*M)$.
\end{remark}

\subsection{Curvature computation}

The curvature of the $\alpha$-connection is given by \cite{zhang}:
$$R^{\hat{\nabla}^{(\alpha)}}(\sigma,\tau)\nu=\frac{1+\alpha}{2}R^{\hat{\nabla}}(\sigma,\tau)\nu+\frac{1-\alpha}{2}R^{\hat{\nabla}^*}(\sigma,\tau)\nu
+(1-\alpha^2)\{\hat{T}(\tau,\hat{T}(\sigma,\nu))-\hat{T}(\sigma,\hat{T}(\tau,\nu))\},$$
for all $\sigma,\tau,\nu\in C^{\infty}(TM\oplus T^*M)$, where $\hat{T}:=\frac{1}{2}(\hat{\nabla}^*-\hat{\nabla})$ is given by:
$$\hat{T}(X+\eta,Y+\beta)=\frac{1}{2}\{h^{-1}(\nabla_X(h(Y)))-\nabla_XY+\nabla_X\beta-h(\nabla_X(h^{-1}(\beta)))\}=$$
$$=\frac{1}{2}\{h^{-1}((\nabla_Xh)(Y))+(\nabla_Xh)(h^{-1}(\beta))\},$$
for all $X,Y\in C^{\infty}(TM)$ and $\eta,\beta\in C^{\infty}(T^*M)$.

\bigskip

We have proved that the curvatures of $\hat{\nabla}$, ${\hat {\nabla}}^*$ and $\nabla$ satisfy \cite{blanan}:
$$R^{\hat {\nabla}}(X+\eta,Y+\beta)(Z+\gamma)=R^{\nabla}(X,Y)Z+h(R^{\nabla}(X,Y)(h^{-1}(\gamma)))$$
and
$$R^{{\hat {\nabla}}^*}(X+\eta,Y+\beta)(Z+\gamma)=h^{-1}(R^{\nabla}(X,Y)(h(Z)))+R^{\nabla}(X,Y)\gamma,$$
for all $X,Y,Z\in C^{\infty}(TM)$ and $\eta,\beta,\gamma\in C^{\infty}(T^*M)$, and we obtain:

\begin{proposition}
Let $\nabla$ be an affine connection on $M$, $h$ a non-degenerate $(0,2)$-tensor field and let $(\check h,\hat{\nabla})$ be the generalized structure on $TM\oplus T^*M$ induced by $(h,\nabla)$. Then the curvature of the $\alpha$-connection defined by $(\hat{\nabla},\hat{\nabla}^*)$ is given by:
$$R^{\hat{\nabla}^{(\alpha)}}(X+\eta,Y+\beta)(Z+\gamma)=
\frac{1+\alpha}{2}R^{\nabla}(X,Y)Z+\frac{1-\alpha}{2}h^{-1}(R^{\nabla}(X,Y)(h(Z)))+
$$
$$+\frac{1-\alpha^2}{4}\{h^{-1}(\nabla_Y\nabla_X(h(Z)))-h^{-1}(\nabla_Y(h(\nabla_XZ)))-\nabla_Y(h^{-1}(\nabla_X(h(Z))))+\nabla_Y\nabla_XZ-$$
$$-h^{-1}(\nabla_X\nabla_Y(h(Z)))+h^{-1}(\nabla_X(h(\nabla_YZ)))+\nabla_X(h^{-1}(\nabla_Y(h(Z))))-\nabla_X\nabla_YZ\}+$$
$$+\frac{1+\alpha}{2}h(R^{\nabla}(X,Y)(h^{-1}(\gamma)))+\frac{1-\alpha}{2}R^{\nabla}(X,Y)\gamma+$$
$$+\frac{1-\alpha^2}{4}\{\nabla_Y\nabla_X\gamma-\nabla_Y(h(\nabla_X(h^{-1}(\gamma))))-h(\nabla_Y(h^{-1}(\nabla_X\gamma)))+h(\nabla_Y\nabla_X(h^{-1}(\gamma)))-$$
$$-\nabla_X\nabla_Y\gamma+\nabla_X(h(\nabla_Y(h^{-1}(\gamma))))+h(\nabla_X(h^{-1}(\nabla_Y\gamma)))-h(\nabla_X\nabla_Y(h^{-1}(\gamma)))\},$$
for all $X,Y,Z\in C^{\infty}(TM)$ and $\eta,\beta,\gamma\in C^{\infty}(T^*M)$.
\end{proposition}

Remark that in terms of $\nabla h$, the expression of $R^{\hat{\nabla}^{(\alpha)}}$ becomes:
$$R^{\hat{\nabla}^{(\alpha)}}(X+\eta,Y+\beta)(Z+\gamma)=
\frac{1+\alpha}{2}R^{\nabla}(X,Y)Z+\frac{1-\alpha}{2}h^{-1}(R^{\nabla}(X,Y)(h(Z)))+$$
$$+\frac{1-\alpha^2}{4}\{h^{-1}((\nabla_{[Y,X]}h)(Z))-h^{-1}((\nabla_Yh)(\nabla_XZ))+h^{-1}((\nabla_Xh)(\nabla_YZ))+$$
$$+h^{-1}(h(R^{\nabla}(X,Y)\cdot,Z))+R^{\nabla}(X,Y)Z-\nabla_Y(h^{-1}((\nabla_Xh)(Z)))+\nabla_X(h^{-1}((\nabla_Yh)(Z)))\}+$$
$$+\frac{1+\alpha}{2}h(R^{\nabla}(X,Y)(h^{-1}(\gamma)))+\frac{1-\alpha}{2}R^{\nabla}(X,Y)\gamma+$$
$$+\frac{1-\alpha^2}{4}\{(\nabla_Yh)(h^{-1}((\nabla_Xh)(h^{-1}(\gamma))))-(\nabla_Xh)(h^{-1}((\nabla_Yh)(h^{-1}(\gamma))))\},$$
for all $X,Y,Z\in C^{\infty}(TM)$ and $\eta,\beta,\gamma\in C^{\infty}(T^*M)$ and we get:

\begin{corollary}\label{c}
If $\nabla h=0$, then
$$R^{\hat{\nabla}^{(\alpha)}}(X+\eta,Y+\beta)(Z+\gamma)=
\frac{1+\alpha}{2}R^{\nabla}(X,Y)Z+\frac{1-\alpha}{2}h^{-1}(R^{\nabla}(X,Y)(h(Z)))+$$
$$+\frac{1-\alpha^2}{4}\{h^{-1}(h(R^{\nabla}(X,Y)\cdot,Z))+R^{\nabla}(X,Y)Z\}+$$
$$+\frac{1+\alpha}{2}h(R^{\nabla}(X,Y)(h^{-1}(\gamma)))+\frac{1-\alpha}{2}R^{\nabla}(X,Y)\gamma,$$
for any $\alpha\in \mathbb{R}$.
\end{corollary}

If $h$ is symmetric and $\nabla h=0$, we have:
$$h(R^{\nabla}(X,Y)\cdot,Z)=-h(R^{\nabla}(X,Y)Z, \cdot)=-h(R^{\nabla}(X,Y)Z),$$
hence:
\begin{corollary}\label{d}
If $h$ is symmetric and $\nabla h=0$, then
$$R^{\hat{\nabla}^{(\alpha)}}(X+\eta,Y+\beta)(Z+\gamma)=R^{\nabla}(X,Y)Z+R^{\nabla}(X,Y)\gamma,$$
for any $\alpha\in \mathbb{R}$.
\end{corollary}

Also, from $$T^{\hat{\nabla}}(X+\eta,Y+\beta)=T^{\nabla}(X,Y)-\{(\nabla_Xh)(h^{-1}(\beta))-(\nabla_Yh)(h^{-1}(\eta))\}$$
and
$$(\hat{\nabla}\check{h})(X+\eta,Y+\beta,Z+\gamma)=(\nabla_Xh)(Y,Z)+(\nabla_Xh)(h^{-1}(\beta),h^{-1}(\gamma)),$$
under the assumption $T^{\nabla}=0$ and $\nabla h=0$, we get $T^{\hat{\nabla}}=0$ and $\hat{\nabla}\check{h}=0$ and we can state:

\begin{corollary}
If $\nabla$ is the Levi-Civita connection of the pseudo-Riemannian metric $h$ on $M$, then the induced connection $\hat{\nabla}$ on $TM\oplus T^*M$ can be called the Levi-Civita connection of $\check{h}$ and its curvature
is given by:
$$R^{\hat{\nabla}}(X+\eta,Y+\beta)(Z+\gamma)=R^{\nabla}(X,Y)Z+R^{\nabla}(X,Y)\gamma,$$
for all $X,Y,Z\in C^{\infty}(TM)$ and $\eta,\beta,\gamma\in C^{\infty}(T^*M)$.
\end{corollary}

\bigskip

Let $\{E_i\}_{1\leq i\leq n}$ be an orthonormal frame field on $M$ with respect to $h$ and consider the orthonormal frame field on $TM\oplus T^*M$ with respect to $\check{h}$ defined by $\{\frac{1}{\sqrt{2}}(E_i+h(E_i)),\frac{1}{\sqrt{2}} (E_i-h(E_i))\}_{1\leq i\leq n}.$ Then the Ricci curvature of $\hat{\nabla}^{(\alpha)}$ is:
$$\Ric^{\hat{\nabla}^{(\alpha)}}(Y+\beta,Z+\gamma):=$$
$$:=\frac{1}{2}\sum_{i=1}^n\{ \check{h}(R^{\hat{\nabla}^{(\alpha)}}(E_i+h(E_i),Y+\beta)(Z+\gamma), E_i+h(E_i))+$$$$+\check{h}(R^{\hat{\nabla}^{(\alpha)}}(E_i-h(E_i),Y+\beta)(Z+\gamma), E_i-h(E_i))\}=$$
$$=\sum_{i=1}^n\{ \check{h}(R^{\hat{\nabla}^{(\alpha)}}(E_i,Y+\beta)(Z+\gamma), E_i)+\check{h}(R^{\hat{\nabla}^{(\alpha)}}(h(E_i),Y+\beta)(Z+\gamma), h(E_i))\}=$$
$$=\sum_{i=1}^n \check{h}(R^{\hat{\nabla}^{(\alpha)}}(E_i,Y+\beta)(Z+\gamma), E_i)=$$
$$=\frac{1+\alpha}{2}\Ric^{\nabla}(Y,Z)+\frac{1-\alpha}{2}\sum_{i=1}^n(R^{\nabla}(E_i,Y)(h(Z)))(E_i)+$$
$$+\frac{1-\alpha^2}{4}\sum_{i=1}^n\{(R^{\nabla}(Y,E_i)(h(Z)))(E_i)+({\nabla}_{[Y,E_i]}(h(Z)))(E_i)-h({\nabla}_Y(h^{-1}({\nabla}_{E_i}(h(Z)))),E_i)+$$
$$+h(R^{\nabla}(Y,E_i)Z,E_i)+h({\nabla}_{[Y,E_i]}Z,E_i)-({\nabla}_Y(h({\nabla}_{E_i}Z)))(E_i)+$$
$$+({\nabla}_{E_i}(h({\nabla}_YZ)))(E_i)+h({\nabla}_{E_i}(h^{-1}({\nabla}_Y(h(Z)))),E_i)\}.$$

If $h$ is symmetric, we get:
$$\sum_{i=1}^n(R^{\nabla}(E_i,Y)(h(Z)))(E_i)=-\sum_{i=1}^nh(R^{\nabla}(E_i,Y)E_i,Z),$$
and we can state:

\begin{proposition}\label{p}
Let $\nabla$ be an affine connection on $M$, $h$ a non-degenerate $(0,2)$-tensor field and let $(\check h,\hat{\nabla})$ be the generalized structure on $TM\oplus T^*M$ induced by $(h,\nabla)$. If $h$ is symmetric, then the Ricci curvature of the $\alpha$-connection defined by $(\hat{\nabla},\hat{\nabla}^*)$ is given by:
$$\Ric^{\hat{\nabla}^{(\alpha)}}(Y+\beta,Z+\gamma)=\left(\frac{1+\alpha}{2}\right)^2\Ric^{\nabla}(Y,Z)+
\left(\frac{1-\alpha}{2}\right)^2\sum_{i=1}^n(R^{\nabla}(E_i,Y)(h(Z)))(E_i)+$$$$+\frac{1-\alpha^2}{4}\sum_{i=1}^n
\{({\nabla}_{[Y,E_i]}(h(Z)))(E_i)-({\nabla}_Y(h({\nabla}_{E_i}Z)))(E_i)+({\nabla}_{E_i}(h({\nabla}_YZ)))(E_i)+$$
$$+h({\nabla}_{[Y,E_i]}Z,E_i)-h({\nabla}_Y(h^{-1}({\nabla}_{E_i}(h(Z)))),E_i)+h({\nabla}_{E_i}(h^{-1}({\nabla}_Y(h(Z)))),E_i)\},$$
for all $Y,Z\in C^{\infty}(TM)$ and $\beta,\gamma\in C^{\infty}(T^*M)$.
\end{proposition}

Remark that in terms of $\nabla h$, the expression of $\Ric^{\hat{\nabla}^{(\alpha)}}$ becomes:
$$\Ric^{\hat{\nabla}^{(\alpha)}}(Y+\beta,Z+\gamma)=\frac{1+\alpha}{2}\Ric^{\nabla}(Y,Z)+\frac{1-\alpha}{2}\sum_{i=1}^n(R^{\nabla}(E_i,Y)(h(Z)))(E_i)+$$
$$+\frac{1-\alpha^2}{4}\sum_{i=1}^n\{({\nabla}_{[Y,E_i]}h)(Z,E_i)-({\nabla}_Yh)({\nabla}_{E_i}Z,E_i)+({\nabla}_{E_i}h)({\nabla}_YZ,E_i)-$$
$$-h({\nabla}_Y(h^{-1}(({\nabla}_{E_i}h)(Z))),E_i)+h({\nabla}_{E_i}(h^{-1}(({\nabla}_Yh)(Z))),E_i)\},$$
for all $Y,Z\in C^{\infty}(TM)$ and $\beta,\gamma\in C^{\infty}(T^*M)$.

\medskip

If $h$ is symmetric and $\nabla h=0$, we have:
$$\Ric^{\nabla}(Y,Z):=\sum_{i=1}^nh(R^{\nabla}(E_i,Y)Z,E_i)=$$
$$=-\sum_{i=1}^nh(R^{\nabla}(E_i,Y)E_i,Z)=\sum_{i=1}^n(R^{\nabla}(E_i,Y)(h(Z)))(E_i),$$
hence:

\begin{corollary} \label{c1} If $h$ is symmetric and $\nabla h=0$, then
$$\Ric^{\hat{\nabla}^{(\alpha)}}(Y+\beta,Z+\gamma)=\Ric^{\nabla}(Y,Z)$$
and
$$\scal^{(\check{h}, \hat{\nabla}^{(\alpha)})}=\scal^{(h,\nabla)},$$
for any $\alpha\in \mathbb{R}$, hence the Ricci curvature tensor field of $\hat{\nabla}^{(\alpha)}$ is symmetric if and only if the Ricci curvature tensor field of $\nabla$ is symmetric.
\end{corollary}

\bigskip

Recall that a manifold $M$ with a dualistic structure $(h;\nabla, \nabla^*)$ is called \textit{conjugate Ricci-symmetric} \cite{Min} if $\Ric^{\nabla}=\Ric^{\nabla^*}$. From Corollary \ref{c1} we can state:

\begin{proposition}
Let $\nabla$ be an affine connection on $M$, $h$ a non-degenerate $(0,2)$-tensor field and let $(\check h,\hat{\nabla})$ be the generalized structure on $TM\oplus T^*M$ induced by $(h,\nabla)$. If $h$ is symmetric and $\nabla h=0$, then $$\Ric^{\hat{\nabla}^{(\alpha)}}(Y+\beta,Z+\gamma)=\Ric^{\nabla}(Y,Z)=\Ric^{\hat{\nabla}^{(-\alpha)}}(Y+\beta,Z+\gamma),$$
for any $Y,Z\in C^{\infty}(TM)$, $\beta,\gamma\in C^{\infty}(T^*M)$ and any $\alpha\in \mathbb{R}$, i.e. $(TM\oplus T^*M, \check{h},\hat{\nabla}^{(\alpha)},\hat{\nabla}^{(-\alpha)})$ is a conjugate Ricci-symmetric manifold.
\end{proposition}

On the other hand, an affine connection on $M$ is called \textit{equiaffine} \cite{ns} if it admits a parallel volume form on $M$.
It is known that \cite{ns} the necessary and sufficient condition for a torsion-free affine connection to be equiaffine is that the Ricci tensor to be symmetric. For a given quasi-statistical structure $(h,\nabla)$, since the dual connection $\hat \nabla^*$ of $\hat \nabla$ is torsion-free, from Proposition \ref{p}, we can state:

\begin{proposition}
Let $(h,\nabla)$ be a quasi-statistical structure on $M$ with $h$ symmetric. Then the dual connection $\hat{\nabla}^{*}$ of the generalized connection $\hat \nabla$ induced by $(h,\nabla)$ is an equiaffine connection on $TM\oplus T^*M$ if and only if
$$\sum_{i=1}^n\{h(R^{\nabla}(Y,E_i)E_i,Z)-h(R^{\nabla}(Z,E_i)E_i,Y)\}=0,$$
for any $Y,Z\in C^{\infty}(TM)$.
\end{proposition}

\section{Integrability of generalized structures with respect to $[\cdot,\cdot]_{\nabla^{(\alpha)}}$}

\subsection{$\nabla^{(\alpha)}$-bracket}
Let $h$ be a non-degenerate symmetric or skew-symmetric $(0,2)$-tensor field on the smooth manifold $M$, let $\nabla$ be a torsion-free affine connection on $M$ such that $(h,\nabla)$ is a statistical structure.

The bracket induced by $\nabla^*$ is given by:
$$[X+\eta,Y+\beta]_{\nabla^{*}}=[X+\eta,Y+\beta]_{\nabla}-\{(\nabla_Xh)(h^{-1}(\beta))-(\nabla_Yh)(h^{-1}(\eta))\}.$$

Let $[\cdot,\cdot]_{\nabla^{(\alpha)}}$ be the bracket defined by the $\alpha$-connection $\nabla^{(\alpha)}$. Then:
$$[X+\eta,Y+\beta]_{\nabla^{(\alpha)}}:=[X,Y]+\nabla^{(\alpha)}_X\beta-\nabla^{(\alpha)}_Y\eta=$$
$$=[X+\eta,Y+\beta]_{\nabla}-\frac{1-\alpha}{2}\{(\nabla_Xh)(h^{-1}(\beta))-(\nabla_Yh)(h^{-1}(\eta))\},$$
for any $X,Y\in C^{\infty}(TM)$ and $\eta,\beta\in C^{\infty}(T^*M)$.

\subsection{$h$-symmetric $(1,1)$-tensor fields}
Let $h$ be a non-degenerate symmetric $(0,2)$-tensor field and let $\nabla$ be a torsion-free affine connection on $M$. Let $J:TM \rightarrow TM$ be a $h$-symmetric $(1,1)$-tensor field, that is, $h(JX,Y)=h(X,JY)$, for any $X,Y\in C^{\infty}(TM)$ and let $(J^*\eta)(X):=\eta(JX)$, for any $X\in C^{\infty}(TM)$ and $\eta\in C^{\infty}(T^*M)$.

We consider the tensor field $F$ defined by:
$$F(X,Y,Z):=h((\nabla_X J)Y,Z),$$
for any $X, Y,Z \in C^{\infty}(TM)$,
which is very important in the classification of almost complex structures. In the following, we will relate some properties of $F$ to the theory of $\alpha$-connections. First of all, we prove the following:

\begin{lemma} For any $X, Y \in C^{\infty}(TM)$, the following holds:
 $$(\nabla_{X}h)(JY)-J^*((\nabla_{X}h)(Y))=-F(X,Y,\cdot)+F(X,\cdot,Y).$$
\end{lemma}
\begin{proof} We have:
$$((\nabla_{X}h)(JY)-J^*((\nabla_{X}h)(Y)))(T)=$$
$$=X(h(JY,T))-h(\nabla_X JY,T)-h(JY,\nabla_X T)-X(h(Y,JT))+h(\nabla_X Y,JT)+h(Y,\nabla_X JT)=$$
$$=-h((\nabla_X J)Y,T)+h(Y,(\nabla_X J)T)=F(X,T,Y)-F(X,Y,T),$$
for any $X, Y, T \in C^{\infty}(TM).$
\end{proof}
Moreover:
\begin{lemma} If $(M,h,\nabla)$ is a statistical manifold, then:
$$((\nabla_{X}h)(JY)-(\nabla_{Y}h)(JX)=0$$
if and only if
$$F(X,\cdot,Y)-F(X,Y,\cdot)-F(Y,\cdot,X)+F(Y,X,\cdot)=0,$$
for any $X, Y \in C^{\infty}(TM).$
\end{lemma}
\begin{proof} From the previous lemma, we get:
$$(\nabla_{X}h)(JY)-(\nabla_{Y}h)(JX)=J^*((\nabla_{X}h)(Y)-(\nabla_{Y}h)(X))+$$
$$+F(X,\cdot,Y)-F(X,Y,\cdot)-F(Y,\cdot,X)+F(Y,X,\cdot)$$
and by the statistical hypothesis, we have the statement.
\end{proof}

Finally:
\begin{lemma} If $(M,h,\nabla)$ is a statistical manifold, then:
$$(\nabla_{X}h)(J^2Y)=(\nabla_{JY}h)(JX)+F(X,\cdot,JY)-F(X,JY,\cdot)-F(JY,\cdot,X)+F(JY,X,\cdot),$$
for any $X, Y \in C^{\infty}(TM)$.
\end{lemma}
\begin{proof} It follow immediately from previous lemmas.
\end{proof}

\begin{remark} Let  $J:TM \rightarrow TM$ be a $h$-symmetric $(1,1)$-tensor field on $M$. Then $\nabla J$ is $h$-symmetric, that is, $h((\nabla_XJ)Y,Z)=h((\nabla_XJ)Z,Y)$, for any $X,Y,Z \in C^{\infty}(TM)$, if and only if $F$ satisfies:
$$F(X,Y,Z)=F(X,Z,Y),$$
for any $X,Y,Z \in C^{\infty}(TM)$.
\end{remark}

\begin{remark} Let $M$ be a smooth manifold with an affine connection $\nabla$ and let $h$ be a symmetric $\nabla$-parallel $(0,2)$-tensor field on $M$. Then for any $h$-symmetric $(1,1)$-tensor field $J$ on $M$, we have that $\nabla J$ is $h$-symmetric.
\end{remark}

Moreover, by a direct computation, we easily get:

\begin{proposition} If $(M,h,\nabla)$ is a statistical manifold and $J:TM \rightarrow TM$ is a $h$-symmetric $(1,1)$-tensor field on $M$ such that $\nabla J$ is $h$-symmetric, then ${\nabla}^* J$ is $h$-symmetric and furthermore, ${\nabla}^{(\alpha)} J$ is $h$-symmetric.

\end{proposition}

Moreover:
\begin{proposition} If $(M,h,\nabla)$ is a statistical manifold and $J:TM \rightarrow TM$ is a $h$-symmetric $(1,1)$-tensor field on $M$, then $\nabla J={\nabla}^* J$ if and only if $\nabla J$ is $h$-symmetric.

\end{proposition}
\begin{proof} We have:
$$({\nabla}_X^* J)Y=({\nabla}_X J)Y+h^{-1}\{F(X,\cdot,Y)-F(X,Y,\cdot)\},$$
for any $X, Y \in C^{\infty}(TM).$ Then the statement.
\end{proof}
\begin{corollary} If $(M,h,\nabla)$ is a statistical manifold and $J:TM \rightarrow TM$ is a $h$-symmetric $(1,1)$-tensor field on $M$, then $\nabla J=0$ if and only if ${\nabla}^* J=0$ and furthermore, if $\nabla J=0$, then ${\nabla}^{(\alpha)} J=0$, for any $\alpha\in \mathbb{R}$, in particular, $\nabla^{(0)} J=0$, where $\nabla ^{(0)}$ is the Levi-Civita connection of $h$.
\end{corollary}

\subsection{Integrability of generalized almost product structures}

\begin{proposition}\label{bb} Let $h$ be a non-degenerate symmetric $(0,2)$-tensor field and let $\nabla$ be a torsion-free affine connection on $M$. Let us suppose that  $(M,h,\nabla)$ is a statistical manifold and let $J:TM \rightarrow TM$ be a $h$-symmetric $(1,1)$-tensor field. Let $\hat{J}:=\left(
                    \begin{array}{cc}
                      J & (I-J^2)h^{-1} \\
                      h & -J^* \\
                    \end{array}
                  \right)
$ be the generalized almost product structure induced by $(h,J)$. Then
 the Nijenhuis tensor field of $\hat{J}$ with respect to $[\cdot,\cdot]_{\nabla^{(\alpha)}}$ satisfies:
$$N_{\hat{J}}^{\nabla^{(\alpha)}}(X+\eta,Y+\beta)=N_{\hat{J}}^{\nabla}(X+\eta,Y+\beta),$$ for any $X,Y\in C^{\infty}(TM)$, $\eta,\beta \in C^{\infty}(T^*M)$ and for any $\alpha\in \mathbb{R}$ if and only if the tensor field $F$ satisfies the following condition:
$$F(X,Y,Z)+F(Y,Z,X)-F(X,Z,Y)-F(Y,X,Z)=0,$$
for any $X,Y,Z\in C^{\infty}(TM)$.
\end{proposition}
\begin{proof} Let us suppose $\eta=h(Z)$ and $\beta=h(W)$ for some $Z,W\in C^{\infty}(TM)$. By using the previous lemmas, we get:
$$N_{\hat{J}}^{\nabla^{(\alpha)}}(X,Y)-N_{\hat{J}}^{\nabla}(X,Y)=$$
$$=-\frac{1-\alpha}{2}\{(\nabla_{JX}h)(Y)-(\nabla_{JY}h)(X)+J^*((\nabla_{X}h)(Y))-J^*((\nabla_{Y}h)(X))\}=$$
$$=-\frac{1-\alpha}{2}\{(\nabla_{Y}h)(JX)-(\nabla_{X}h)(JY)+J^*((\nabla_{X}h)(Y))-J^*((\nabla_{Y}h)(X))\}=$$
$$=-\frac{1-\alpha}{2}\{-h((\nabla_{Y}J)(X), \cdot)+h(X,(\nabla_{Y}J)(\cdot))+h((\nabla_{X}J)(Y), \cdot)-h(Y,(\nabla_{X}J)(\cdot))\}=$$
$$=-\frac{1-\alpha}{2}\{-F(Y,X,\cdot)+F(Y,\cdot,X)+F(X,Y,\cdot)-F(X,\cdot,Y)\}.$$

$$N_{\hat{J}}^{\nabla^{(\alpha)}}(h(Z),h(W))-N_{\hat{J}}^{\nabla}(h(Z),h(W))=$$
$$=-\frac{1-\alpha}{2}\{(\nabla_{(J^2-I)Z}h)(JW)-(\nabla_{(J^2-I)W}h)(JZ)-J^*((\nabla_{(J^2-I)Z}h)(W))+$$
$$+J^*((\nabla_{(J^2-I)W}h)(Z))+(J^2-I)h^{-1}(-(\nabla_{(J^2-I)Z}h)(W)+(\nabla_{(J^2-I)W}h)(Z))\}=$$
$$=-\frac{1-\alpha}{2}\{(\nabla_{JW}h)(J^2Z)-(\nabla_{JZ}h)(J^2W)-J^*((\nabla_{J^2Z}h)(W))+J^*((\nabla_{J^2W}h)(Z))-$$
$$-h((\nabla_{Z}J)(W), \cdot)+h(W,(\nabla_{Z}J)(\cdot))+h((\nabla_{W}J)(Z), \cdot)-h(Z,(\nabla_{W}J)(\cdot))\}-$$
$$-\frac{1-\alpha}{2}\{(J^2-I)h^{-1}(-(\nabla_{J^2Z}h)(W)+(\nabla_{J^2W}h)(Z))\}=$$
$$=-\frac{1-\alpha}{2}\{F(JW,\cdot,JZ)-F(JW,JZ,\cdot)-F(JZ,\cdot,JW)+F(JZ,JW,\cdot)-$$
$$-F(Z,W,\cdot)+F(Z,\cdot,W)+F(W,Z,\cdot)-F(W,\cdot,Z)\}-$$
$$-\frac{1-\alpha}{2}\{(J^2-I)h^{-1}(-(\nabla_{W}h)(J^2Z)+(\nabla_{Z}h)(J^2W))\}=$$
$$=-\frac{1-\alpha}{2}\{F(JW,\cdot,JZ)-F(JW,JZ,\cdot)-F(JZ,\cdot,JW)+F(JZ,JW,\cdot)-$$
$$-F(Z,W,\cdot)+F(Z,\cdot,W)+F(W,Z,\cdot)-F(W,\cdot,Z)\}-$$
$$-\frac{1-\alpha}{2}\{(J^2-I)h^{-1}(-F(W,\cdot,JZ)+F(W,JZ,\cdot)+F(JZ,\cdot,W)-F(JZ,W,\cdot)+$$
$$+F(Z,\cdot,JW)-F(Z,JW,\cdot)-F(JW,\cdot,Y)+F(JW,Z,\cdot))\}.$$

$$N_{\hat{J}}^{\nabla^{(\alpha)}}(X,h(W))-N_{\hat{J}}^{\nabla}(X,h(W))=$$
$$=-\frac{1-\alpha}{2}\{-(\nabla_{JX}h)(JW)+(\nabla_{(J^2-I)W}h)(X)+$$$$+J^*((\nabla_{JX}h)(W))-J^*((\nabla_{X}h)(JW))+(\nabla _ X h)(W)\}-$$
$$-\frac{1-\alpha}{2}\{(J^2-I)h^{-1}((\nabla_{JX}h)(W)-(\nabla_{X}h)(JW))\}=$$
$$=-\frac{1-\alpha}{2}\{-(\nabla_{JX}h)(JW)+(\nabla_{X}h)(J^2W)+J^*((\nabla_{JX}h)(W))-J^*((\nabla_{X}h)(JW))\}-$$
$$-\frac{1-\alpha}{2}\{(J^2-I)h^{-1}((\nabla_{W}h)(JX)-(\nabla_{X}h)(JW))\}=$$
$$=-\frac{1-\alpha}{2}\{F(X,\cdot,JW)-F(X,JW,\cdot)-F(JW,\cdot,X)+F(JW,X,\cdot)+$$
$$+J^*(F(W,\cdot,X)-F(W,X,\cdot)-F(X,\cdot,W)+F(X,W,\cdot))\}-$$
$$-\frac{1-\alpha}{2}\{(J^2-I)h^{-1}((F(W,\cdot,X)-F(W,X,\cdot)-F(X,\cdot,W)+F(X,W,\cdot))\}.$$
Then the proof is complete.
\end{proof}

\begin{corollary} Let $(M,h,\nabla)$ be a statistical manifold and let $J:TM \rightarrow TM$ be a $h$-symmetric $(1,1)$-tensor field such that $\nabla J$ is $h$-symmetric. Then the Nijenhuis tensor field of the generalized almost product structure $\hat{J}$ with respect to $[\cdot,\cdot]_{\nabla^{(\alpha)}}$ satisfies:
$$N_{\hat{J}}^{\nabla^{(\alpha)}}(X+\eta,Y+\beta)=N_{\hat{J}}^{\nabla}(X+\eta,Y+\beta),$$ for any $X,Y\in C^{\infty}(TM)$, $\eta,\beta \in C^{\infty}(T^*M)$ and for any $\alpha \in \mathbb{R}$. In particular, the definition of integrability for $\hat J$ is $\alpha$-invariant.
\end{corollary}

\subsection{Integrability of generalized almost complex structures}

Analogously we have the following.

\begin{proposition}\label{dd} Let $h$ be a non-degenerate symmetric $(0,2)$-tensor field and let $\nabla$ be a torsion-free affine connection on $M$. Let us suppose that  $(M,h,\nabla)$ is a statistical manifold and let $J:TM \rightarrow TM$ be a $h$-symmetric $(1,1)$-tensor field. Let $\hat{J}_-:=\left(
                    \begin{array}{cc}
                      J & -(I+J^2)h^{-1} \\
                      h & -J^* \\
                    \end{array}
                  \right)
$ be the generalized almost complex structure induced by $(h,J)$. Then
 the Nijenhuis tensor field of $\hat{J}_-$ with respect to $[\cdot,\cdot]_{\nabla^{(\alpha)}}$ satisfies:
$$N_{\hat{J}_-}^{\nabla^{(\alpha)}}(X+\eta,Y+\beta)=N_{\hat{J}_-}^{\nabla}(X+\eta,Y+\beta),$$ for any $X,Y\in C^{\infty}(TM)$, $\eta,\beta \in C^{\infty}(T^*M)$ and for any $\alpha\in \mathbb{R}$ if and only if the tensor field $F$ satisfies the following condition:
$$F(X,Y,Z)+F(Y,Z,X)-F(X,Z,Y)-F(Y,X,Z)=0,$$
for any $X,Y,Z\in C^{\infty}(TM)$.
\end{proposition}
\begin{proof} As before, let us suppose $\eta=h(Z)$ and $\beta=h(W)$ for some $Z,W\in C^{\infty}(TM)$. We get:
$$N_{\hat{J}_-}^{\nabla^{(\alpha)}}(X,Y)-N_{\hat{J}}^{\nabla}(X,Y)=$$
$$=-\frac{1-\alpha}{2}\{(\nabla_{JX}h)(Y)-(\nabla_{JY}h)(X)+J^*((\nabla_{X}h)(Y))-J^*((\nabla_{Y}h)(X))\}=$$
$$=-\frac{1-\alpha}{2}\{(\nabla_{Y}h)(JX)-(\nabla_{X}h)(JY)+J^*((\nabla_{X}h)(Y))-J^*((\nabla_{Y}h)(X))\}=$$
$$=-\frac{1-\alpha}{2}\{-h((\nabla_{Y}J)(X), \cdot)+h(X,(\nabla_{Y}J)(\cdot))+h((\nabla_{X}J)(Y), \cdot)-h(Y,(\nabla_{X}J)(\cdot))\}=$$
$$=-\frac{1-\alpha}{2}\{-F(Y,X,\cdot)+F(Y,\cdot,X)+F(X,Y,\cdot)-F(X,\cdot,Y)\}.$$

$$N_{\hat{J}_-}^{\nabla^{(\alpha)}}(h(Z),h(W))-N_{\hat{J}}^{\nabla}(h(Z),h(W))=$$
$$=-\frac{1-\alpha}{2}\{(\nabla_{(J^2+I)Z}h)(JW)-(\nabla_{(J^2+I)W}h)(JZ)-$$$$-J^*((\nabla_{(J^2+I)Z}h)(W))+J^*((\nabla_{(J^2+I)W}h)(Z))\}+$$
$$+(J^2+I)h^{-1}(-(\nabla_{(J^2+I)Z}h)(W)+(\nabla_{(J^2+I)W}h)(Z))\}=$$
$$=-\frac{1-\alpha}{2}\{-h((\nabla_{Z}J)(W), \cdot)+h(W,(\nabla_{Z}J)(\cdot))+h((\nabla_{W}J)(Z), \cdot)-h(Z,(\nabla_{W}J)(\cdot))\}+$$
$$+(J^2+I)h^{-1}(-(\nabla_{J^2 Z}h)(W)+(\nabla_{J^2W}h)(Z))\}=$$
$$=-\frac{1-\alpha}{2}\{-h((\nabla_{J^2 Z}J)(W), \cdot)+h(W,(\nabla_{J^2Z}J)(\cdot))+h((\nabla_{J^2W}J)(Z), \cdot)-h(Z,(\nabla_{J^2 W}J)(\cdot))\}=$$
$$=-\frac{1-\alpha}{2}\{F(JW,\cdot,JZ)-F(JW,JZ,\cdot)-F(JZ,\cdot,JW)+F(JZ,JW,\cdot)-$$
$$-F(Z,W,\cdot)+F(Z,\cdot,W)+F(W,Z,\cdot)-F(W,\cdot,Z)\}-$$
$$-\frac{1-\alpha}{2}\{(J^2-I)h^{-1}(-F(W,\cdot,JZ)+F(W,JZ,\cdot)+F(JZ,\cdot,W)-F(JZ,W,\cdot)+$$
$$+F(Z,\cdot,JW)-F(Z,JW,\cdot)-F(JW,\cdot,Y)+F(JW,Z,\cdot))\}.$$

$$N_{\hat{J}_-}^{\nabla^{(\alpha)}}(X,h(W))-N_{\hat{J}}^{\nabla}(X,h(W))=$$
$$=-\frac{1-\alpha}{2}\{-(\nabla_{JX}h)(JW)+(\nabla_{(J^2+I)W}h)(X)+$$$$+J^*((\nabla_{JX}h)(W))-J^*((\nabla_{X}h)(JW))-(\nabla _ X h)(W)\}-$$
$$-\frac{1-\alpha}{2}\{(J^2+I)h^{-1}((\nabla_{JX}h)(W)-(\nabla_{X}h)(JW))\}=$$
$$=-\frac{1-\alpha}{2}\{-(\nabla_{JX}h)(JW)+(\nabla_{X}h)(J^2W)+J^*((\nabla_{JX}h)(W))-J^*((\nabla_{X}h)(JW))\}-$$
$$-\frac{1-\alpha}{2}\{(J^2+I)h^{-1}((\nabla_{W}h)(JX)-(\nabla_{X}h)(JW))\}=$$
$$=-\frac{1-\alpha}{2}\{F(X,\cdot,JW)-F(X,JW,\cdot)-F(JW,\cdot,X)+F(JW,X,\cdot)+$$
$$+J^*(F(W,\cdot,X)-F(W,X,\cdot)-F(X,\cdot,W)+F(X,W,\cdot))\}-$$
$$-\frac{1-\alpha}{2}\{(J^2+I)h^{-1}((F(W,\cdot,X)-F(W,X,\cdot)-F(X,\cdot,W)+F(X,W,\cdot))\}.$$
Then the proof is complete.
\end{proof}

\begin{corollary}  Let $(M,h,\nabla)$ be a statistical manifold and let $J:TM \rightarrow TM$ be a $h$-symmetric $(1,1)$-tensor field such that $\nabla J$ is $h$-symmetric. Then the Nijenhuis tensor field of the generalized almost complex structure $\hat{J}_-$ with respect to $[\cdot,\cdot]_{\nabla^{(\alpha)}}$ satisfies:
$$N_{\hat{J}_-}^{\nabla^{(\alpha)}}(X+\eta,Y+\beta)=N_{\hat{J}_-}^{\nabla}(X+\eta,Y+\beta),$$ for any $X,Y\in C^{\infty}(TM)$, $\eta,\beta \in C^{\infty}(T^*M)$ and for any $\alpha \in \mathbb{R}$. In particular, the definition of integrability for $\hat{J}_-$ is $\alpha$-invariant.
\end{corollary}

\subsection{Integrability of generalized metallic structures}

In the metallic case we have the following.

\begin{proposition} Let $h$ be a non-degenerate symmetric $(0,2)$-tensor field and let $\nabla$ be a torsion-free affine connection on $M$. Let us suppose that  $(M,h,\nabla)$ is a statistical manifold and let $J:TM \rightarrow TM$ be a $(1,1)$-tensor field such that $J$ and $\nabla J$ are $h$-symmetric. Then
 the Nijenhuis tensor field of the generalized metallic structure $\hat{J}:=
 \left(
                    \begin{array}{cc}
                      -J+pI & (-J^2+pJ+qI)h^{-1} \\
                      h & J^* \\
                    \end{array}
                  \right)
 $, $p,q\in \mathbb{R}$, with respect to $[\cdot,\cdot]_{\nabla^{(\alpha)}}$ satisfies:
$$N_{\hat{J}}^{\nabla^{(\alpha)}}(X+\eta,Y+\beta)=N_{\hat{J}}^{\nabla}(X+\eta,Y+\beta),$$ for any $X,Y\in C^{\infty}(TM)$, $\eta,\beta \in C^{\infty}(T^*M)$ and for any $\alpha\in \mathbb{R}$.  In particular, the definition of integrability for $\hat J$ is $\alpha$-invariant.
\end{proposition}
\begin{proof} Remember that:
$$N_{\hat{J}}^{\nabla^{(\alpha)}}(X+\eta,Y+\beta)=[\hat{J}(X+\eta),\hat{J}(Y+\beta)]_{\nabla^{(\alpha)}}-$$
$$-\hat{J}[\hat{J}(X+\eta),Y+\beta]_{\nabla^{(\alpha)}}-\hat{J}[(X+\eta,\hat{J}(Y+\beta)]_{\nabla^{(\alpha)}}+{\hat{J}}^2([X+\eta,Y+\beta]_{\nabla^{(\alpha)}})$$
and that $${\hat{J}}^2=p\hat J +qI.$$
As before, let us suppose $\eta=h(Z)$ and $\beta=h(W)$ for some $Z,W\in C^{\infty}(TM)$. We get:
$$N_{\hat{J}}^{\nabla^{(\alpha)}}(X,Y)-N_{\hat{J}}^{\nabla}(X,Y)=$$
$$=-\frac{1-\alpha}{2}\{(\nabla_{-JX+pX}h)(Y)-(\nabla_{-JY+pY}h)(X)-\hat J(((\nabla_{X}h)(Y))-((\nabla_{Y}h)(X)))\}=$$
$$=-\frac{1-\alpha}{2}\{-(\nabla_{Y}h)(JX)+(\nabla_{X}h)(JY)\}=$$
$$=\frac{1-\alpha}{2}\{-h((\nabla_{Y}J)(X), \cdot)+h(X,(\nabla_{Y}J)(\cdot))+h((\nabla_{X}J)(Y), \cdot)-h(Y,(\nabla_{X}J)(\cdot))\}=0.$$

$$N_{\hat{J}}^{\nabla^{(\alpha)}}(h(Z),h(W))-N_{\hat{J}}^{\nabla}(h(Z),h(W))=$$
$$=-\frac{1-\alpha}{2}\{(\nabla_{(-J^2+pJ+qI)Z}h)(JW)-(\nabla_{(-J^2+pJ+qI)W}h)(JZ)-$$
$$-\hat J(((\nabla_{(-J^2+pJ+qI)Z}h)(W))-((\nabla_{(-J^2+pJ+qI)(W)}h)(Z)))\}=$$
$$=-\frac{1-\alpha}{2}\{(\nabla_{-J^2Z}h)(JW)-(\nabla_{-J^2W}h)(JZ)+$$
$$-\hat J(((\nabla_{-J^2Z}h)(W))+p(\nabla_{JZ}h)W-((\nabla_{-J^2W}h)(Z))-p(\nabla_{JW}h)(Z)\}=$$
$$=-\frac{1-\alpha}{2}\{-F(JW,\cdot,JZ)+F(JW,JZ,\cdot)+F(JZ,\cdot,JW)-F(JZ,JW,\cdot)-$$
$$-\hat J(-F(W,\cdot,JZ)+F(W,JZ,\cdot)+F(JZ,\cdot,W)-F(JZ,W,\cdot)+$$
$$+F(Z,\cdot,JW)-F(Z,JW,\cdot)-F(JW,\cdot,Z)-F(JW,Z,\cdot)+$$
$$+F(W,\cdot,Z)-F(W,Z,\cdot)-F(Z,\cdot,W)-F(Z,W,\cdot))\}=0.$$

$$N_{\hat{J}}^{\nabla^{(\alpha)}}(X,h(W))-N_{\hat{J}}^{\nabla}(X,h(W))=$$
$$=-\frac{1-\alpha}{2}\{-(\nabla_{JX}h)(JW)+p(\nabla_{X}h)(JW)-(\nabla_{(-J^2+pJ+qI)W} h)(X)-$$
$$-\hat J((\nabla_{-JX}h)(W)+p((\nabla_{X}h)(W))+(\nabla _ X h)(JW))+p\hat J (\nabla _ X h)(W)+q(\nabla _ X h)(W)\}=$$
$$=-\frac{1-\alpha}{2}\{F(X,\cdot,JW)-F(X,JW,\cdot)-F(JW,\cdot,X)+F(JW,X,\cdot)+$$
$$+\hat J(-F(X,\cdot,W)+F(X,W,\cdot)+F(W,\cdot,X)-F(W,X,\cdot)\}=0.$$
Then the proof is complete.
\end{proof}

\begin{corollary} Let $(M,h,\nabla)$ be a statistical manifold and let $J:TM \rightarrow TM$ be a $(1,1)$-tensor field such that $\nabla J=0$. Then the Nijenhuis tensor field of the generalized metallic structure $\hat{J}$ with respect to $[\cdot,\cdot]_{\nabla^{(\alpha)}}$ satisfies:
$$N_{\hat{J}}^{\nabla^{(\alpha)}}(X+\eta,Y+\beta)=N_{\hat{J}}^{\nabla}(X+\eta,Y+\beta),$$ for any $X,Y\in C^{\infty}(TM)$, $\eta,\beta \in C^{\infty}(T^*M)$ and for any $\alpha \in \mathbb{R}$.
\end{corollary}


\subsection{${\hat \nabla}^{(\alpha)}$-parallel generalized structures}

\begin{lemma}\label{mm}
Let $h$ be a non-degenerate symmetric $(0,2)$-tensor field on $M$, let $J:TM \rightarrow TM$ be a $h$-symmetric $(1,1)$-tensor field and let $\hat{J}$ be either the generalized almost product structure
$\left(
                    \begin{array}{cc}
                      J & (I-J^2)h^{-1} \\
                      h & -J^* \\
                    \end{array}
                  \right)
$ or the generalized almost complex structure $\left(
                    \begin{array}{cc}
                      J & -(I+J^2)h^{-1} \\
                      h & -J^* \\
                    \end{array}
                  \right)
$ induced by $(h,J)$. Then:

i) $\hat{\nabla}\hat{J}=0$ if and only if $\nabla J=0$;

ii) $\hat{\nabla}^*\hat{J}=0$ if and only if $\nabla J=0$.
\end{lemma}

\begin{proof}
For any $X,Y\in C^{\infty}(TM)$ and any $\eta,\beta \in C^{\infty}(T^*M)$, we have:
$$(\hat{\nabla}_{X+\eta}\hat{J})(Y+\beta)=(\nabla_XJ)(Y)-(\nabla_XJ)(J(h^{-1}(\beta)))-J((\nabla_XJ)(h^{-1}(\beta)))-$$$$-h((\nabla_XJ)(h^{-1}(\beta)))).$$

Taking into account that for any $X,W\in C^{\infty}(TM)$ and any $\beta \in C^{\infty}(T^*M)$, we have:
$$((\nabla_XJ^*)(\beta))(W)=\beta((\nabla_XJ)(W)),$$
we obtain:
$$(\hat{\nabla}^*_{X+\eta}\hat{J})(Y+\beta)=h^{-1}((\nabla_XJ^*)(h(Y)))-h^{-1}((\nabla_XJ^*)(J^*\beta))-h^{-1}(J^*((\nabla_XJ^*)(\beta)))-
$$$$-(\nabla_XJ^*)(\beta).$$
Then the proof is complete.
\end{proof}

\begin{lemma}\label{aa}
Let $h$ be a non-degenerate symmetric $(0,2)$-tensor field on $M$, let $J:TM \rightarrow TM$ be a $h$-symmetric $(1,1)$-tensor field and let $\hat{J}$ be the generalized metallic structure $\hat{J}:=
 \left(
                    \begin{array}{cc}
                      -J+pI & (-J^2+pJ+qI)h^{-1} \\
                      h & J^* \\
                    \end{array}
                  \right)
 $, $p,q\in \mathbb{R}$, induced by $(h,J)$. Then:

i) $\hat{\nabla}\hat{J}=0$ if and only if $\nabla J=0$;

ii) $\hat{\nabla}^*\hat{J}=0$ if and only if $\nabla J=0$.
\end{lemma}

\begin{proof}
For any $X,Y\in C^{\infty}(TM)$ and any $\eta,\beta \in C^{\infty}(T^*M)$, we have:
$$(\hat{\nabla}_{X+\eta}\hat{J})(Y+\beta)=-(\nabla_XJ)(Y)-(\nabla_XJ)(J(h^{-1}(\beta)))-J((\nabla_XJ)(h^{-1}(\beta)))+p(\nabla_XJ)(h^{-1}(\beta))-$$$$-h((\nabla_XJ)(h^{-1}(\beta)))).$$

Taking into account that for any $X,W\in C^{\infty}(TM)$ and any $\beta \in C^{\infty}(T^*M)$, we have:
$$((\nabla_XJ^*)(\beta))(W)=\beta((\nabla_XJ)(W)),$$
we obtain:
$$(\hat{\nabla}^*_{X+\eta}\hat{J})(Y+\beta)=-h^{-1}((\nabla_XJ^*)(h(Y)))-h^{-1}((\nabla_XJ^*)(J^*\beta))-$$$$-h^{-1}(J^*((\nabla_XJ^*)(\beta)))+
ph^{-1}((\nabla_XJ^*)(\beta))+
$$$$+(\nabla_XJ^*)(\beta).$$
Then the proof is complete.
\end{proof}

By a direct computation, for any $X,Y\in C^{\infty}(TM)$ and any $\eta,\beta \in C^{\infty}(T^*M)$, we get:
$$(\hat{\nabla}^{(\alpha)}_{X+\eta}\hat{J})(Y+\beta)=\frac{1+\alpha}{2}(\hat{\nabla}_{X+\eta}\hat{J})(Y+\beta)+
\frac{1-\alpha}{2}(\hat{\nabla}^*_{X+\eta}\hat{J})(Y+\beta),$$
and using the previous two lemmas, we can state:
\begin{proposition}
Let $h$ be a non-degenerate symmetric $(0,2)$-tensor field on $M$, let $J:TM \rightarrow TM$ be a $h$-symmetric $(1,1)$-tensor field and let $\hat{J}$ be either the generalized almost product structure, the generalized almost complex structure from Lemma \ref{mm} or the generalized metallic structure from Lemma \ref{aa}. Then for any $\alpha\in \mathbb{R}$, $\hat{\nabla}^{(\alpha)}\hat{J}=0$ if and only if $\nabla J=0$.
\end{proposition}

\section{$\alpha$-connection of the twin metric}

\subsection{Statistical structure defined by the twin metric}

Let $(M,g)$ be a pseudo-Riemannian manifold, let $\nabla$ be the Levi-Civita connection of $g$ and let $J$ be a $g$-symmetric $(1,1)$-tensor field. We can prove the following.

\begin{lemma} $$g((\nabla_X J)Y,Z)=g((\nabla _X J)Z,Y),$$
for any $X,Y,Z \in C^{\infty}(TM)$.
\end{lemma}
\begin{proof} Since $\nabla$ is the Levi-Civita connection of $g$ and $J$ is $g$-symmetric, we have:
$$g((\nabla_X J)Z,Y)=g(\nabla _X JZ,Y)-g(J(\nabla _X Z),Y)=$$
$$=X(g(JZ,Y))-g(JZ,\nabla_X Y)-g(J(\nabla_X Z),Y)=$$
$$=X(g(Z,JY))-g(Z,J(\nabla_XY))-g(\nabla_X Z,JY)=$$
$$=g(\nabla_X Z,JY)+g(Z,\nabla_X JY)-g(Z,J(\nabla_X Y))-g(\nabla_X Z,JY)=$$
$$=g((\nabla_XJ)Y,Z).$$
\end{proof}

Let us suppose that $J$ is invertible and let $\tilde g$ be the twin metric defined by $g$ and $J$:
$$\tilde g(X,Y):=g(X,JY),$$
for any $X,Y \in C^{\infty}(TM)$.
We have the following.

\begin{lemma} $$(\nabla_X \tilde g)(Y,Z)=g((\nabla_X J)Z,Y),$$
for any $X,Y,Z \in C^{\infty}(TM)$.
\end{lemma}
\begin{proof}
$$(\nabla_X \tilde g)(Y,Z)=X(\tilde g (Y,Z))-\tilde g(\nabla_X Y,Z)-\tilde g(Y,\nabla_X Z)=$$
$$=X(g(Y,JZ))-g(\nabla_XY,JZ)-g(Y,J(\nabla_X Z))=$$
$$=g(\nabla_XY,JZ)+g(Y,\nabla_X JZ)-g(\nabla_X Y,JZ)-g(Y,J(\nabla_X Z))=$$
$$=g((\nabla_XJ)Z,Y).$$
\end{proof}

\begin{proposition} Let $(M,g)$ be a pseudo-Riemannian manifold, let $\nabla$ be the Levi-Civita connection of $g$ and let $J$ be a $g$-symmetric $(1,1)$-tensor field. Assume that $J$ is invertible and let $\tilde g$ be the twin metric defined by $g$ and $J$.
Then $(M,\tilde g,\nabla)$ is a statistical manifold if and only if
$$(\nabla_X J)Y=(\nabla _Y J)X,$$
for any $X,Y \in C^{\infty}(TM)$.
\end{proposition}

\begin{proof} From the previous lemmas, we get:
$$(\nabla_X \tilde g)(Y,Z)-(\nabla_Y \tilde g)(X,Z)=g((\nabla_XJ)Z,Y)-g((\nabla_YJ)Z,X)=$$
$$=g((\nabla_XJ)Y,Z)-g((\nabla_YJ)X,Z)=g((\nabla_XJ)Y-(\nabla_YJ)X,Z).$$
Then the statement.
\end{proof}

\subsection{Dualistic structure defined by the twin metric}

Let $(M,g)$ be a pseudo-Riemannian manifold and let $\nabla$ be the Levi-Civita connection of $g$. Let $J$ be a $g$-symmetric $(1,1)$-tensor field and let $\tilde g$ be the twin metric defined by $g$ and $J$. We have the following.
\begin{proposition} If $J$ is invertible, then $(\tilde g, \nabla)$ define a dualistic structure and the dual connection $\nabla^*$ is given by:
$$\nabla_X ^* Y=\nabla_X Y+J^{-1}((\nabla_X J)Y),$$
for any $X,Y \in C^{\infty}(TM)$.
\end{proposition}
\begin{proof} We have:
$$X(\tilde g (Y,Z))=\tilde g(\nabla_X Y,Z)+\tilde g(Y,\nabla_X^* Z)$$
if and only if
$$X(g(Y,JZ))-g(\nabla_XY,JZ)=g(Y,J(\nabla^*_X Z)).$$
Then
$$g(Y,\nabla_X JZ)=g(Y,J(\nabla_X^*Z)),$$
for any $X,Y,Z \in C^{\infty}(TM)$. Hence $J(\nabla_X^*Y)=\nabla_X JY$, or:
$$\nabla_X ^* Y=\nabla_X Y+J^{-1}((\nabla_X J)Y).$$
\end{proof}

Moreover, a direct computation gives the following.
\begin{corollary}
$$\nabla_X ^* \eta=\nabla_X \eta-(J^{-1}(\nabla_X J))^*\eta,$$
for any $X \in C^{\infty}(TM)$ and for any $\eta \in C^{\infty}(T^*M).$
\end{corollary}

In particular, the $\alpha$-connection defined by the twin metric has the following expression:
$$\nabla^{(\alpha)}=\nabla-\frac{1-\alpha}{2}J^{-1}(\nabla J).$$

\subsection{Generalized geometry for the dualistic structure defined by the twin metric}

Let $(M,g)$ be a pseudo-Riemannian manifold, let $\nabla$ be the Levi-Civita connection of $g$ and let $J$ be a $g$-symmetric $(1,1)$-tensor field such that
$$(d^{\nabla}J)(X,Y):=(\nabla_X J)Y-(\nabla_Y J)X=0,$$
for any $X,Y \in C^{\infty}(TM)$.

Let $(\tilde g, \nabla)$ be the statistical structure on $M$ defined by $g$ and $J$ and let $(\tilde g, \nabla,\nabla ^*)$ be the corresponding dualistic structure.

Remark that from the condition $d^\nabla J=0$ it follows that the tensor field $F$ defined by
$$F(X,Y,Z):=\tilde g ((\nabla_X J)Y,Z)$$ satisfies:
$$F(X,Y,Z)= \tilde g ((\nabla_X J)Y,Z)=\tilde g ((\nabla_Y J)X,Z)=F(Y,X,Z),$$
for any $X,Y,Z \in C^{\infty}(TM).$

As a consequence, we can restate Propositions \ref{bb} and \ref{dd} as follows.
\begin{proposition} Let $\hat{J}:=\left(
                    \begin{array}{cc}
                      J & (I-J^2){\tilde g}^{-1} \\
                      \tilde g & -J^* \\
                    \end{array}
                  \right)
$ be the generalized almost product structure induced by $(\tilde g,J)$. Then
 the Nijenhuis tensor field of $\hat{J}$ with respect to $[\cdot,\cdot]_{\nabla^{(\alpha)}}$ satisfies:
$$N_{\hat{J}}^{\nabla^{(\alpha)}}(X+\eta,Y+\beta)=N_{\hat{J}}^{\nabla}(X+\eta,Y+\beta),$$ for any $X,Y\in C^{\infty}(TM)$, $\eta,\beta \in C^{\infty}(T^*M)$ and for any $\alpha\in \mathbb{R}$ if and only if the tensor field $F$ satisfies the following condition:
$$F(Y,Z,X)-F(X,Z,Y)=0,$$
for any $X,Y,Z\in C^{\infty}(TM)$.
\end{proposition}

\begin{proposition} Let $\hat{J}_-:=\left(
                    \begin{array}{cc}
                      J & -(I+J^2){\tilde g}^{-1} \\
                      \tilde g & -J^* \\
                    \end{array}
                  \right)
$ be the generalized almost complex structure induced by $(\tilde g,J)$. Then  the Nijenhuis tensor field of $\hat{J}_-$ with respect to $[\cdot,\cdot]_{\nabla^{(\alpha)}}$ satisfies:
$$N_{\hat{J}_-}^{\nabla^{(\alpha)}}(X+\eta,Y+\beta)=N_{\hat{J}_-}^{\nabla}(X+\eta,Y+\beta),$$ for any $X,Y\in C^{\infty}(TM)$, $\eta,\beta \in C^{\infty}(T^*M)$ and for any $\alpha\in \mathbb{R}$ if and only if the tensor field $F$ satisfies the following condition:
$$F(Y,Z,X)-F(X,Z,Y)=0,$$
for any $X,Y,Z\in C^{\infty}(TM)$.
\end{proposition}

Under the above hypothesis, a direct computation gives the following.
\begin{lemma}\label{xx} $F(Y,\cdot,X)-F(X,\cdot,Y)=0$ if and only if $$J((\nabla_XJ)Y)=(\nabla_X J)JY$$
or equivalently, if and only if $$J((\nabla_Y J)X)=(\nabla_{JY} J)X,$$
for any $X,Y\in C^{\infty}(TM)$.
\end{lemma}
\begin{remark} The geometrical meaning of the last condition in Lemma \ref{xx} appears in Proposition 16, \cite{an}. Precisely, for the generalized complex structure
 $\hat{J}_-:=\left(
                    \begin{array}{cc}
                      J & 0 \\
                      0 & -J^* \\
                    \end{array}
                  \right)
$,
the Nijenhuis tensor field of $\hat{J}_-$ with respect to $\nabla$, $N^\nabla_{\hat{J}_-}$, coincide with the Nijenhuis tensor field of $\hat{J}_-$ with respect to the Courant bracket in $TM\oplus T^* M$, $N^c_{\hat{J}_-}$, if and only if $J((\nabla_Y J)X)=(\nabla_{JY} J)X$, for any $X,Y\in C^{\infty}(TM)$.
\end{remark}

If $J$ is a metallic structure on $M$, then we have the following.

\begin{proposition} Let $(M,g)$ be a pseudo-Riemannian manifold, let $\nabla$ be the Levi-Civita connection of $g$ and let $J$ be a $g$-symmetric $(1,1)$-tensor field on $M$ such that $J^2=pJ+qI$, for some $p,q$ real numbers with $p^2+4p\neq 0$. Then $J((\nabla_XJ)Y)=(\nabla_X J)JY$, for any $X,Y\in C^{\infty}(TM)$ if and only if $\nabla J=0$.
\end{proposition}
\begin{proof} We have:
$$(\nabla_X J)JY-J((\nabla_X J)Y)=\nabla_X J^2Y-J(\nabla_X JY)-J(\nabla_X JY-J(\nabla_XY))=$$
$$=\nabla_X J^2Y-2J((\nabla_X J)Y)-J^2(\nabla_X Y)=(\nabla_X J^2)Y-2J((\nabla_X J)Y)=$$
$$=p(\nabla_X J)Y-2J((\nabla_X J)Y)=(pI-2J)((\nabla_X J)Y).$$
In particular, if $\frac{p}{2}$ is not an eigenvalue of $J$, or equivalently, if $p^2+4p\neq 0$, then we get $\nabla J=0$.
\end{proof}

\bigskip

\textit{Adara M. Blaga}

\textit{Department of Mathematics}

\textit{West University of Timi\c{s}oara}

\textit{Bld. V. P\^{a}rvan nr. 4, 300223, Timi\c{s}oara, Rom\^{a}nia}

\textit{adarablaga@yahoo.com}

\bigskip

\textit{Antonella Nannicini}

\textit{Department of Mathematics and Informatics "U. Dini"}

\textit{University of Florence}

\textit{Viale Morgagni, 67/a, 50134, Firenze, Italy}

\textit{antonella.nannicini@unifi.it}


\begin{thebibliography}{99}

\bibitem{am} S.-I. Amari, \textit{Information Geometry and Its Applications}, Springer, Tokyo, Japan, 2016.




\bibitem{bn} {A. M. Blaga, A. Nannicini}, \textit{Generalized metallic structures}, Rev. Union Mat. Argentina \textbf{61}, no. 1 (2020), 73--86.


\bibitem{blanan} {A. M. Blaga, A. Nannicini}, \textit{Generalized quasi-statistical structures}, arXiv:1809.04784, 2018.









\bibitem{Min} {C. Min, W. Ri, K. Kwak, D. An}, \textit{Equiaffine structure and conjugate Ricci-symmetry of a statistical manifold}, Differ. Geom. Appl. \textbf{41} (2015), 39-47.

\bibitem{an} {A. Nannicini}, \textit{Almost complex structures on cotangent bundles and generalized geometry}, J. Geom. Phys. \textbf{60} (2010), 1781-1791.





\bibitem{ns} {K. Nomizu, T. Sasaki}, \textit{Affine differential geometry}, Cambridge Univ. Press (1994).


\bibitem{sh} H. Shima, \textit{The Geometry of Hessian Structures}, World Scientific, Singapore, 2007.


\bibitem{zhang} J. Zhang, \textit{A note on curvature of $\alpha$-connections of a statistical manifold}, AISM \textbf{59} (2007), 161-170.

\end{thebibliography}
\end{document}